\documentclass[11pt,letterpaper]{amsart}
\usepackage[utf8]{inputenc}
\usepackage{amsmath}
\usepackage{amsfonts}
\usepackage{amssymb}
\usepackage{mathrsfs}
\usepackage{graphicx}

\usepackage[left=3cm,right=3cm, top=3cm,bottom=3cm,marginparwidth=24mm]{geometry}

\usepackage[pdfencoding=auto, psdextra, pagebackref, colorlinks=true, linkcolor=blue, citecolor=blue, urlcolor=blue]{hyperref}

\usepackage{tikz-cd} 
\usepackage[all,cmtip]{xy}
\xyoption{arrow}
\usepackage{xcolor,etoolbox}
\usepackage{enumitem}

\usepackage{bm}
\usepackage{mathtools}

\usepackage{yhmath}

\newcommand{\bigslant}[2]{{\raisebox{.2em}{$#1$}\left/\raisebox{-.2em}{$#2$}\right.}}

\usepackage{enumitem}

\newcommand{\Qp}{{Q_{+}}}
\newcommand{\Qm}{{Q_{-}}}
\newcommand{\Qpm}{{Q_{{\pm}}}}

\newcommand{\Ll}{\mathcal{L}}

\newcommand{\Ct}{\widetilde{C}}

\newcommand{\Res}{\text{Res}}

\newcommand{\mf}[1]{\mathfrak #1}
\newcommand{\mc}[1]{\mathcal #1}
\newcommand{\ms}[1]{\mathscr #1}

\newcommand{\im}{\operatorname{im}}

\newcommand{\ov}[1]{\overline{#1}}

\newcommand{\LCP}{(C, P_{\bullet})}
\newcommand{\LCtPQ}{(\widetilde{C}, P_{\bullet} \cup Q_\pm)}
\newcommand{\FV}{\mathbb{V}}

\theoremstyle{plain}

\newtheorem{thm}[subsection]{Theorem}
\newtheorem{lem}[subsection]{Lemma}
\newtheorem{cor}[subsection]{Corollary}
\newtheorem{prop}[subsection]{Proposition}
\newtheorem*{thm*}{Theorem}
\newtheorem*{rem*}{Remark}
\newtheorem*{lem*}{Lemma}

\newtheorem*{cor*}{Corollary}
\newtheorem*{prop*}{Proposition}

\theoremstyle{definition}
\newtheorem{defn}[subsection]{Definition}

\newtheorem{question}[subsection]{Question}
\newtheorem{remark}[subsection]{Remark}

\newcommand{\ZZ}{\mathbb{Z}}

\newcommand{\NN}{\mathbb{N}}
\newcommand{\CC}{\mathbb{C}}

\newcommand{\VV}{\mathbb{V}}

\newcommand{\UV}{\mathscr{U}(V)}

\newcommand{\LV}{\mathfrak{L}(V)}
\newcommand{\Ac}{\mathcal{A}}
\newcommand{\Lc}{\mathcal{L}}

\newcommand{\SV}{\mathscr{S}\!\!\mathscr{M}}

\newcommand{\Mgn}[1][1]{{\mathcal{M}_{g, #1}}}
\newcommand{\bMgn}[1][1]{{\overline{\mathcal{M}}_{g,#1}}}

\newcommand{\tMgn}[1][1]{{\widetriangle{\mathcal{M}}_{g,#1}}}

\newcommand{\Oo}{\mathscr{O}}

\usepackage{cleveref}

\begin{document}
\title[Factorization presentations]{Factorization presentations}


\begin{abstract}
Modules over a vertex operator algebra $V$ give rise to sheaves of coinvariants on moduli of stable pointed curves.
If $V$ satisfies finiteness and semi-simplicity conditions, these sheaves are vector bundles. This relies on factorization, an isomorphism of spaces of coinvariants at a nodal curve with a finite sum of analogous spaces on the normalization of the curve. Here we introduce the notion of a factorization presentation, and using this, we show that finiteness conditions on $V$ imply the sheaves of coinvariants are coherent on moduli spaces of pointed stable curves without any assumption of semisimplicity.
\end{abstract}

{
\author[C.~Damiolini]{Chiara Damiolini}
\address{Chiara Damiolini \newline \indent  Department of Mathematics, University of Pennsylvania,  Phil, PA 08904}
\email{chiara.damiolini@gmail.com}}
{
\author[A.~Gibney]{Angela Gibney}
\address{Angela Gibney \newline  \indent  Department of Mathematics, University of Pennsylvania,  Phil, PA 08904}
\email{angela.gibney@gmail.com}}
{
\author[D.~Krashen]{Daniel Krashen}
\address{Daniel Krashen \newline  \indent  Department of Mathematics, University of Pennsylvania,  Phil, PA 08904}
\email{daniel.krashen@gmail.com}}

\maketitle
There are sheaves of coinvariants (and dual sheaves of conformal blocks) on $\overline{\mathcal{M}}_{g,n}$, the moduli stack parametrizing families of Deligne-Mumford stable pointed curves of genus $g$. These are defined by representations of certain vertex operator algebras (called VOAs for short). VOAs generalize commutative associative algebras as well as Lie algebras, and have played important roles in both mathematics and physics, in understanding conformal field theories, finite group theory, and in the construction of knot invariants and 3-manifold invariants.  

For some time, such sheaves were known to be defined on smooth pointed curves with coordinates \cite{bzf}, and in special cases, or in low genus, on curves with nodes  \cite{tuy, bfm,  NT}. The inspirational first and best understood example is given by representations of the affine vertex operator algebra $V_\ell(\mathfrak{g})$, and its simple quotient $L_\ell(\mathfrak{g})$, derived from a simple Lie algebra $\mathfrak{g}$, and $\ell \in \mathbb{C}$. For $\ell \in \mathbb{Z}_{>0}$, they support a projectively flat connection (with singularities on the boundary), and are vector bundles,  with ranks given by Verlinde formulae \cite{tuy}.  An argument was made that these algebraic structures were coordinate free \cite{tsu}, and so defined on $\overline{\mathcal{M}}_{g,n}$ (see also \cite[\S 8]{dgt2}).  Sometimes referred to as sheaves of covacua, Chern classes were shown to be tautological in \cite{mop, moppz}, where they are called {\em{Verlinde bundles}}.  On $\overline{\mathcal{M}}_{0,n}$ they are globally generated \cite{fakhr}. 

{\em{Strongly rational}} VOAs, satisfying finiteness and semi-simplicity assumptions as in \S \ref{sec:FiniteDefs}, give rise to {\em{generalized Verlinde bundles}} \cite{NT, dgt1, dgt2, dgt3, DG}.   Fibers at nodal curves satisfy the important factorization theorem \cite{NT, dgt2}, which has proved crucial to showing they are locally free of finite rank, sharing many properties with Verlinde bundles. Here we ask whether sheaves defined by {\em{strongly finite}} VOAs, for which semi-simplicity is not required, retain such features. Our main result, in Theorem \ref{thm:FactorizationRes}, provides a modification of factorization, enabling one to prove the sheaves have finite rank (Corollary \ref{cor:CoherentMain}).

\smallskip
To describe these results, we give a small amount of notation (with details given in \S \ref{sec:background}).

To a vertex operator algebra $V$, by \cite{ZhuMod} there is an associative algebra $A(V)$, and a functor taking $A(V)$-modules $W_0$ to $V$-modules $W=\mathcal{M}(W_0)$. Given a stable curve $C$ with $n$ smooth marked points $P_{\bullet}=(P_1,\ldots,P_n)$, coordinates $t_\bullet=(t_1,\ldots, t_n)$, and an assignment of an $A(V)$-module $W^i_0$ to each point $P_i$, one can associate a vector space 
of coinvariants, the largest quotient of $W^\bullet := \otimes_i W^i$ on which a natural Lie algebra acts trivially.  By \cite{dgt1, dgt2}, as $(C, P_{\bullet}, t_\bullet)$ varies, these determine a sheaf on the moduli space of of stable pointed coordinatized curves of genus $g$. If $V$ is $C_2$-cofinite, then by \cite{dgt2}, these sheaves descend to $\overline{\mathcal{M}}_{g,n}$. For simplicity, vector spaces of coinvariants are denoted $[W^\bullet]_{\LCP}$. 

An essential tool in understanding  $[W^\bullet]_{\LCP}$  is factorization, which allows one to reinterpret vector spaces of coinvariants on nodal curves in terms of coinvariants on a curve with fewer singularities. To describe this, suppose that the curve $C$ has only one node $Q \in C$, and let $\widetilde{C} \to C$ be the normalization of $C$, with points $\Qp$ and $\Qm$ lying over $Q$. To define coinvariants on $\widetilde{C}$, one has the $A(V)$ module $W^i$ for each point $P_{i}$, and it is natural to assign a \textit{single} $A(V)$-\textit{bimodule} $E$ to the pair of markings $\Qp$ and $\Qm$ of $\widetilde{C}$. One may then define the vector space of coinvariants
$[W^\bullet \otimes \Phi (B)]_{(\widetilde{C},P_{\bullet}\cup \Qpm)}$, where $\Phi$ is a functor taking  $A(V)$-bimodules  to the $\UV^2$-modules, compatible with the functor $\mathcal{M}$  (see \S \ref{sec:Zhu} for the universal enveloping algebra $\UV$, and \S \ref{sec:Phi} for $\Phi$).   Proposition \ref{prop:AbstractFactorMain}, our main technical tool,  asserts that if $V$ is $C_1$-cofinite (which implies that $A(V)$ is finitely generated), the coinvariants
$[W^\bullet]_{\LCP}$ and $[W^\bullet \otimes \Phi (A(V))]_{(\widetilde{C},P_{\bullet}\cup \Qpm)}$ coincide (here $A(V)$ is considered as a bimodule over itself).

Although Proposition \ref{prop:AbstractFactorMain} rephrases the coinvariants $[W^\bullet]_{(C, P_{\bullet})}$ in terms of something associated to a less singular curve $\widetilde C$, the expression $[W^\bullet \otimes \Phi( A(V))]_{(\widetilde{C},P_{\bullet}\cup \Qpm)}$ is, a priori, of a somewhat different nature, as we have associated a single bimodule as opposed to a pair of $A(V)$ modules to the points $\Qp$ and $\Qm$. 
On the other hand, by Theorem \ref{thm:FactorizationRes}, when an $A(V)$ bimodule $E$ is \textit{factorizable}, that is, if it can be written as a sum of the form $E = \bigoplus (X^+_0 \otimes X^-_0)$ (see Definition~\ref{def:FR}), we may identify the spaces
\[[W^\bullet \otimes \Phi( E)]_{(\widetilde{C},P_{\bullet}\cup \Qpm)} = 
\bigoplus [W^\bullet \otimes X^+ \otimes X^{-}]_{(\widetilde{C},P_{\bullet}\cup \Qpm)}
\]
with the vector space of coinvariants obtained by assigning the modules $W^\bullet$ to the points $P_{\bullet}$, and assigning $X^\pm$ to $\Qpm$.  
In particular,  we can use this to rewrite $[W^\bullet]_{(C, P_{\bullet})}$ in terms of coinvariants on $\widetilde C$ whenever $A(V)$ is factorizable.
For $V$ rational and $C_2$-cofinite (which implies that $A(V)$ is finite and semi-simple) then  $A(V) = \bigoplus (X_0 \otimes X^{\vee}_0)$, a finite sum over all simple $A(V)$-modules $X_0$, and this recovers  \cite[Factorization Theorem]{dgt2}: \begin{equation}\label{eq:VFact}
[W^\bullet]_{\LCP} \cong \bigoplus_{}[W^\bullet \otimes X \otimes X^{\vee}]_{(\widetilde{C},P_{\bullet}\cup \Qpm)},\end{equation}
a finite sum, indexed by the isomorphism classes of all simple $V$-modules $X$.

It turns out that this is in some sense sharp---an associative algebra $A$ is isomorphic to a finite direct sum of tensor products of left and right $A$-modules if and only if $A$ is finite and semi-simple (see \S \ref{sec:SaltyBell}).  It follows that for naturally occurring VOAs for which $A(V)$ is not semi-simple, but does satisfy finiteness conditions (such as being finitely generated or finite dimensional), other approaches are needed to relate coinvariants on nodal curves to coinvariants on curves with fewer singularities.

Our strategy is to observe that if $V$ is $C_1$-cofinite, then $A(V)$ has what we call a factorization resolution  $\cdots  \overset{\alpha}{\rightarrow} \oplus \left(X_0 \otimes Y_0 \right) \rightarrow A(V) \rightarrow 0$ (see Definition \ref{def:FR}). In Theorem \ref{thm:FactorizationRes}, we show that from any such factorization resolution of $A(V)$, one obtains a factorization presentation of nodal coinvariants. In particular, Theorem \ref{thm:FactorizationRes} expresses coinvariants at nodal curves as a quotient of a sum of coinvariants on the normalization (as in~\eqref{eq:VFact}).  However in this case, the sum, which may not be finite, is indexed by indecomposable $V$-modules. The factorization presentation of Theorem \ref{thm:FactorizationRes} specializes to \eqref{eq:VFact} if $V$ is rational and $C_2$-cofinite, giving an alternative proof.

The proof of Theorem \ref{thm:FactorizationRes}, involves two steps: First an application of Proposition \ref{prop:AbstractFactorMain}, mentioned earlier, which asserts that if $V$ is $C_1$-cofinite, one has a natural isomorphism
\begin{equation}\label{eq:StepOne}
[W^\bullet]_{(C, P_{\bullet})}\cong
[W^\bullet \otimes \Phi (A(V))]_{(\Ct, P_{\bullet} \cup \Qpm)}.
\end{equation}
 In the second step, the right hand side of \eqref{eq:StepOne} shown to be the cokernel of a right exact functor applied to a factorization resolution of $A(V)$.

Two consequences of Theorem \ref{thm:FactorizationRes} are given for $C_2$-cofinite $V$  (which implies that $A(V)$ is finite dimensional). In this case $A(V)$ has a unique bimodule decomposition as a finite sum of principal indecomposable $A(V)$-bimodules (Lemma \ref{lem:BimodDecomp}).  By Corollary \ref{cor:FactorSum}, 
\begin{equation}\label{eq:Two}[W^\bullet]_{\LCP}\cong
\bigoplus_{} [W^\bullet \otimes X \otimes X']_{(\widetilde{C},P_{\bullet}\cup \Qpm)}
\oplus \bigoplus [W^\bullet \otimes \Phi( I)]_{(\widetilde{C},P_{\bullet}\cup \Qpm)},\end{equation}
where $X$ and $X'$ are dual simple $V$-modules obtained by applying the functor $\Phi$ to specified simple indecomposable bimodules in its bimodule decomposition, and $\Phi(I)$ is a tensor product of indecomposable $V$-modules given by the remaining principal indecomposable bimodules $I$.  In \S \ref{sec:Examples}, this is illustrated for the triplet $\mathcal{W}(p)$ and the Symplectic Fermions $F(d)^+$,  important families of strongly finite, non-rational VOAs.  One may refine  \eqref{eq:Two} via a factorization resolution of the indecomposable bimodules $I$ (see Definition \ref{def:FR}), as we demonstrate for the $\mathcal{W}(p)$.

The second consequence of Theorem \ref{thm:FactorizationRes} is that sheaves of coinvariants defined by representations of $C_2$-cofinite VOAs are coherent on $\ov{\mc{M}}_{g,n}$ (Corollary \ref{cor:CoherentMain}). Vector spaces of coinvariants at smooth pointed coordinatized curves were shown to be finite dimensional in this generality in \cite[Proposition 5.1.1]{dgt2}, based on arguments made for a related construction in \cite{AN1}.  The result in Corollary \ref{cor:CoherentMain} improves \cite[Theorem 8.4.2.]{dgt2} which concludes that spaces of coinvariants at nodal curves are also finite dimensional if $V$ is both $C_2$-cofinite and rational. Corollary  \ref{cor:CoherentMain} achieves the first step towards showing that the sheaves we consider may form vector bundles (see \S \ref{sec:Questions} for a discussion of the problem).

A vertex operator algebra $V$ is $C_1$ cofinite if and only if it is (strongly) generated in finite degree. There are strongly finitely generated vertex operator algebras which are not rational or $C_2$-cofinite, and in this case fibers of the sheaf coinvariants at a point $(C,P_\bullet)$ should be regarded as dependent on tangent data to the curve $C$ at the marked points $P_i$.  The affine VOAs $V_{\ell}(\mathfrak{g})$, and $L_{\ell}(\mathfrak{g})$ are defined for all $\ell \in \mathbb{C}$ with $\ell\ne -h^{\vee}$. They are generated in degree $1$ (so are $C_1$-cofinite), but are $C_2$-cofinite if and only if  $\ell$ is a positive integer. The Virasoro vertex operator algebras $\textit{Vir}_c$ for $c\in \mathbb{C}$ are generated in degree $2$, but if $c$ is not in the discrete series, they are not rational or $C_2$-cofinite.  If the sums in the numerator of Theorem \ref{thm:FactorizationRes} are not finite for such examples, they cannot be used (as we do in the $C_2$-cofinite case), to prove that the nodal coinvariants are finite dimensional, but nevertheless Theorem \ref{thm:FactorizationRes} applies.

A VOA which is rational and $C_2$ cofinite is called strongly rational if it is also simple and self-dual. As mentioned, a sheaf of coinvariants $\mathbb{V}(V;W^\bullet)$ defined by a rational and $C_2$-cofinite VOA $V$  is a vector bundle \cite{dgt2}. Assumptions that $V$ is also simple and self-dual enable one to show that  Chern classes of $\mathbb{V}(V;W^\bullet)$ lie in the tautological ring, as was proved for simple affine $V=L_\ell(\mathfrak{g})$ for $\ell \in \mathbb{Z}_{>0}$ in \cite{mop,moppz}, and more generally, for $V$ strongly rational  in \cite{dgt3}. It is not clear that Chern classes of vector bundles defined by strongly finite VOAs would be tautological (see \S \ref{sec:Questions}). 

Interest in sheaves of coinvariants (and dual sheaves of conformal blocks)
originates from mathematical physics.  In \cite{HuangBook} a thorough historical account of the connection between conformal field theory and vertex operator algebras is given. Observations made by Witten in \cite{WittenJones},  stimulating new developments in several directions, pointed to an unexpected bridge between Verlinde bundles (their correlation functions,  and the WZW, or Hitchin connection), and knot invariants (like the Jones polynomial).  In algebraic geometry, there was great interest in natural isomorphisms between vector spaces of conformal blocks and generalized theta functions  \cite{BertramSzenes, Bertram, Thaddeus, bl1, Faltings, P, LaszloSorger, BGFin, BF1}.

The work mentioned above has led to the understanding that there are correspondences between (A) modules over {\it{strongly rational VOAs}}, (B) rational conformal field theories, and (C) certain aspects of generalized Verlinde bundles.  It is natural to ask how the relationships between (A), (B), and (C) may be generalized if strongly rational VOAs are replaced by more general VOAs. Modules over {\it{strongly finite VOAs}} are expected to correspond to logarithmic conformal field theories (see \cite{HuangLog} for an account).  This analogy, supported by a number of examples (see \cite{CGPitch} for a discussion), brings us to investigate whether the coherent sheaves of coinvariants defined by strongly finite VOAs  are locally free, giving  vector bundles.

To show that such sheaves of coinvariants are vector bundles, one can try to prove an analogue of the sewing theorem, a refined version of factorization, characterizing coinvariants at infinitesimal smoothings of nodal curves.  For $V$ rational and $C_2$-cofinite,  \cite{tuy, NT, dgt2} this says that to a non-trivial deformation of a nodal curve, there corresponds a trivial deformation of the space of coinvariants. A more general sewing statement, while likely different if semi-simplicity does not hold, should provide a measure of consistency, which would for instance reflect, if $\mathbb{V}(V;W^\bullet)$ are vector bundles, that invariants like ranks are constant in families.  Sewing statements, similarly described, have been made by researchers looking at such questions in the non-semisimple context, from other perspectives \cite{HuangLog, consistency}.  

\smallskip
 \noindent
{\em{Plan of paper:}} In \S \ref{sec:background} our notation is given,
in \S \ref{sec:InducedModules} we define the functor $\Phi$, and its relationship to Zhu's functor $\mathcal{M}$. In \S \ref{sec:main} we define factorization resolutions, proving Theorem \ref{thm:FactorizationRes}, and Proposition \ref{prop:AbstractFactorMain}.  Consequences of Theorem \ref{thm:FactorizationRes} are proved in \S \ref{sec:Consequences}, and examples presented in \S \ref{sec:Examples}. In \S \ref{sec:Questions} we discuss the evidence that strongly finite VOAs define even more general Verlinde bundles, what is left to show, and other related open problems for algebraic geometers.

\section{Preliminaries}\label{sec:background}

Throughout this paper, we will work over the complex numbers. That is, we will use the term vector space to mean vector space over the complex numbers and algebra to mean algebra over the complex numbers.

 \subsection{Vertex operator algebras and their modules} We work with finitely generated $\mathbb{N}$-gradable vertex operator algebras $V$ of CFT-type, and an important tool for their study is Zhu's algebra.  We briefly define these and some of their properties (see \cite{fhl,dl} for more information).
\subsubsection{VOAs} A vertex operator algebra of CFT-type, called in this paper \textit{VOA}, is a four-tuple $(V, {\textbf{1}}, \omega, Y(\cdot,z))$, where:
\begin{enumerate}
\item $V=\bigoplus_{i\in \mathbb{N}} V_i$ is a non-negatively graded  vector space with $\dim V_i<\infty$;
\item ${\textbf{1}}$ is an element in $V_0$, called the \textit{vacuum vector} such that $V_0 = \CC \textbf{1}$;
\item $\omega$ is an element in $V_2$, called the \textit{conformal vector};
\item $Y(\cdot,z)\colon V \rightarrow  \textrm{End}(V)[[z,z^{-1}]]$ is a linear operation   $v \mapsto  Y(v,z) :=\sum_{i\in\mathbb{Z}} v_{(i)}z^{-i-1}$. \\The series $Y(v,z)$ is called the \textit{vertex operator} assigned to $v$.
\end{enumerate}
\noindent
This data satisfy a number of axioms.

\subsubsection{The universal enveloping algebra and V-modules}\label{sec:Universal}
There are a number of ways to define $V$-modules.
Following \cite{NT}, $V$-modules are certain
modules over $\mathscr{U}(V)$, the (completed) universal enveloping algebra defined in \cite[\S 1.3]{FrenkelZhu}, called the current algebra in \cite[\S 2.2]{NT}.
This is an associative algebra, topologically generated by the enveloping algebra a Lie algebra $\mathfrak{L}(V)$, described next.

As a vector space, $\mathfrak{L}(V)$ is the quotient $(V\otimes \mathbb{C}(\!(t)\!)) / \im \nabla$, where
\[\nabla: V\otimes \mathbb{C}(\!(t)\!) \rightarrow V\otimes \mathbb{C}(\!(t)\!), \qquad v\otimes f \mapsto L_{-1}v \otimes f + v \otimes \frac{df}{dt}.\] Here $L_{-1}=\omega_{(0)}$ is the coefficent of $z^{-1}$ in the power series $Y(\omega, z)=\sum_{m\in \mathbb{Z}}\omega_{(m)}z^{-m-1}$.  This operator is like a derivative, since from the axioms for a VOA, for any $v\in V$, one has $Y(L_{-1}v,z)=\frac{d}{dz}Y(v,z)$.  The bracket for the Lie algebra $\mathfrak{L}(V)$ is defined on generators $v_{[i]}=\overline{v\otimes t^j}$, and $u_{[j]}=\overline{u\otimes t^j}$ by Borcherd's identity, also a consequence of the axioms:
\[ [v_{[i]},u_{[j]}]=
\sum_{k=0}^{\infty}\binom{i}{k}(v_{(k)}(u))_{[i+j-k]}.\]

There is a triangular decomposition $\mathfrak{L}(V) = \mathfrak{L}(V)_{< 0} \oplus \mathfrak{L}(V)_{0} \oplus \mathfrak{L}(V)_{> 0}$, where
\[\mathfrak{L}(V)_{< 0}  = {\text{Span}}\left\lbrace v_{[i]} \in \mathfrak{L}(V) \, :\,  {\deg}(v_{[i]})={\deg}(v) - i -1 <0 \right\rbrace,\]
\[\mathfrak{L}(V)_{> 0}={\text{Span}}\left\lbrace v_{[i]} \in \mathfrak{L}(V)  \, :\,  {\deg}(v_{[i]})={\deg}(v) - i -1 > 0\right\rbrace, \mbox{ and}\] \[\mathfrak{L}(V)_{0}={\text{Span}} \left\lbrace v_{[{\deg}(v)-1]}  \, :\, v \in V \mbox{ homogeneous}\right\rbrace.\]

By \cite[Definition 2.3.1]{NT} a $V$-module $W$ is a finitely generated $\ms{U}(V)$-module such that for any $w\in W$, the vector space $F^0\ms{U}(V)w$ is a finite-dimensional vector space, and there is a positive integer $d$ such that $F^d\ms{U}(V)w=0$, where the filtration is induced from that of $\mf{L}(V)$ (see eg \cite[\S 2.2]{NT}).  

As is explained in \cite{NT}, the filtration of $\ms{U}(V)$ allows one to show that $V$-modules are $\mathbb{N}$-gradable. These are called admissible modules and grading restricted weak $V$-modules in the literature. We will refer to them  as \textit{$V$-modules}.  

Such modules can also be described as pairs $\left(W, Y^W(-,z)\right)$
consisting of a vector space $W=\bigoplus_{i \in \mathbb{N}} W_i$, with ${\rm{dim}}(W_i)<\infty$ for all $i$, together with  $Y^W(-, z) \colon V \to \operatorname{End}(W)[[z,z^{-1}]]$, a linear map sending an element $v \in V$ to the $End(W)$-valued series $Y^W(v,z) = \sum_{i \in \mathbb{Z}} v_{(i)}^Wz^{-i-1}$,
and for which this data satisfies a number of axioms. 

\subsubsection{Zhu's associative algebra and the functors $\mathcal{M}$ and $L$}\label{sec:Zhu}
There is an associative algebra $A(V)$, introduced in \cite{ZhuMod} as an appropriate quotient of $V$, whose representation theory was shown to be reflective of the representation theory of $V$ via functors $\mathcal{M}$ and $L$. 

To define the functors, following  \cite[Theorem 3.3.5, (4) and (5)]{NT}, let  $\ms{U}(V)$ be the completion of the universal enveloping algebra for the Lie algebra $\mathfrak{L}(V)$, and recall the triangular decomposition of $\mathfrak{L}(V)$ (see \S \ref{sec:Universal}).   We denote by $\ms U(V)_{\le 0}$ the sub-$\ms U(V)$ algebra,  topologically generated by $\mathfrak{L}(V)_{<0}\oplus \mathfrak{L}(V)_{0}$. By \cite[Proposition 3.1]{DongLiMason}, the map $\ms U(V)_{0}\twoheadrightarrow A(V)$ given by $v_{[\deg v -1]} \mapsto v$ is surjective, and so any $A(V)$-module $E$ is a $\ms U(V)_{0}$-module.  This extends to an $\ms U(V)_{\le 0}$-module structure, letting $\mathfrak{L}(V)_{< 0}$ act trivially. One then sets \begin{equation}\label{eq:Verma}\mathcal{M}(E):=\ms U(V)\otimes_{U(V)_{\le 0}}E.\end{equation} If $E$ is simple, then $\mathcal{M}(E)$ has a unique, possibly zero, maximal sub-module $\mathcal{J}(E)$, and \[L(E):=\mathcal{M}(E)/\mathcal{J}(E)\] is simple, realizing the bijection between simple $V$-modules and simple $A(V)$-modules.

\begin{lem}\label{lem:Indecomposable}For any vertex operator algebra V of CFT-type,
if $E$ is an indecomposable $A(V)$-module, then the generalized Verma module 
$\mathcal{M}(E)$ is an indecomposable V-module.
\end{lem}
\begin{proof}
Let $E$ be an indecomposable $A(V)$-module.  
By \cite[Theorem 3.3.5, Part (4)]{NT}, $\mathcal{M}(E)$ is a $V$-module.  Suppose that 
$\mathcal{M}(E) = M^1 \oplus M^2$ is decomposable.  Then for each $i\in \{1,2\}$, one has $M^i \subset \mathcal{M}(E)$ are $\mathbb{N}$-gradable $\ms{U}(V)$-submodules of $\mathcal{M}(E)$.  Taking the degree zero components, 
we obtain $M^1_0\oplus M^2_0=\mathcal{M}(E)_0=E$.  Since $E$ is indecomposable, one has that either $M^1_0=E$ or $M^2_0=E$. Suppose $M^1_0=E$.  Then 
\[\mathcal{M}(E)= \ms U(V)\otimes_{U(V)_{\le 0}}E=\ms U(V) \cdot (1\otimes E) = \ms U(V) \cdot \mc{M}(E)_0  = \ms U(V) \cdot M^1_0 \subset M^1.\]
So $M^1=\mathcal{M}(E)$, and $M^2=0$.

\end{proof}

\begin{remark}\label{rmk:Simple}
We wonder what assumptions suffice so that $\mathcal{M}(S_0)=L(S_0)$ (see  Remark \ref{rmk:Indecomposable}).
\end{remark}

\subsection{Finiteness conditions}\label{sec:FiniteDefs}
In this paper we refer to various standard finiteness conditions.  

We say that $V$ is \textit{$C_2$-cofinite} if and only if ${\dim}(V/C_2(V)) <\infty$, where \[C_2(V):=\mathrm{span}_{\mathbb{C}}\left\{v_{(-2)}u \,:\, v,u \in V\right\}.\]
If $V$ is $C_2$-cofinite, then  $A(V)$ is finite dimensional \cite{GN}.

We say that $V$ is \textit{$C_1$-cofinite} if and only if ${\dim}(V_+/C_1(V)) <\infty$, where
\[V_{+}=\bigoplus_{d\in \mathbb{N}_{\ge 1}}V_d, \mbox{ and } C_1(V)={\rm{Span}}_{\mathbb{C}}\{v_{(-1)}(u), L_{(-1)}(w) \ | \ v, u \in V_+, w \in V\}.\]
If $V$ is $C_2$-cofinite, then it is also $C_1$-cofinite. Moreover, using \cite[Proposition 3.2]{KarelLi}, if $V$ is $C_1$-cofinite, then  $V$  is strongly finitely generated, hence, as stated by Karel and Li, it follows that $A(V)$ is finitely generated.

 $V$ is \textit{rational} if and only if any $V$-module is a finite direct sum of simple $V$-modules. It follows from \cite[Theorem 2.2.3]{ZhuMod} that if $V$ is rational, then $A(V)$ is finite and semi-simple.

A rational and $C_2$-cofinite vertex algebra $V$ is \textit{strongly rational} if $V$ is simple and self-dual.  A $C_2$-cofinite vertex algebra is \textit{strongly finite} if it is also simple and self-dual. 

\begin{remark}\label{rmk:Indecomposable} When $V$ is rational, indecomposable modules are simple, so by Lemma \ref{lem:Indecomposable}, for $E$ a simple indecomposable $A(V)$-module,  $\mathcal{M}(E)$ is also simple and in particular, $\mathcal{M}(E) = L(E)$.  If $A(V)$ is not semisimple, this may or may not be the case (see also \cite[Problem on page 439 and Condition III]{NT}). We are interested in what conditions give that for $V$ strongly finite, and $E$ a nontrivial simple $A(V)$ module, then $\mathcal{M}(E)=L(E)$. For instance, for the $C_2$-cofinite and non-rational triplet algebras,  the generalized Verma modules induced from the irreducible indecomposable $A(\mathcal{W}(p))$-modules $\Lambda(p)_0$, and $\Pi(p)_0$, are simple \cite{AM}. On the other hand, there are super VOAs $V$ and examples of simple $A(V)$-modules $E$ for which $\mathcal{M}(E)$ is indecomposable but
reducible (eg.~affine vertex superalgebras associated to Lie superalgebras whose even subalgebra is not reductive  \cite{CRArch}).  In these examples, the vertex operator algebras are not $C_2$-cofinite (but they are $C_1$-cofinite).  
\end{remark}

\subsection{Spaces and sheaves of coinvariants}
Given a VOA $V$, $n$ $V$-modules $W^1,\dots, W^n$ and a coordinatized stable $n$-marked curve $(C, P_1, \dots, P_n, t_1, \dots, t_n)$, one can construct the space of coinvariants $\VV(V;W^\bullet)_{(C,P_{\bullet},t_\bullet)}$. This is  the biggest quotient of $W^\bullet:= W^1 \otimes \dots \otimes W^n$ on which the \emph{chiral Lie algebra} $\Ll_{C \setminus P_{\bullet}}(V)$ (see \S \ref{sec:chiral} below for more details) acts trivially.  Here we will denote such fibers as follows
\[ \VV(V;W^\bullet)_{(C,P_{\bullet},t_\bullet)}= \left[ W^\bullet \right]_{\LCP}= \dfrac{W^1 \otimes \dots \otimes W^n}{\Ll_{C \setminus P_{\bullet}}(V)(W^1 \otimes \dots \otimes W^n)}.
\]
To emphasize the Lie algebra used, we sometimes write
\[ \left[W^\bullet \right]_{\LCP}= \left[ W^\bullet \right]_{\mathscr{L}_{C \setminus P_{\bullet}}(V)}.\]
This construction holds in families, and defines a quasi-coherent sheaf $\widetriangle{\mathbb{V}}_g(V;W^\bullet)$ on the moduli space $\tMgn[n]$ paramerizing coordinatized stable $n$-marked curves of genus $g$. It further  descends to a sheaf ${\mathbb{V}}^J_g(V;W^\bullet)$ on $\overline{\mathcal{J}}^\times_{g,n}$ parametrizing $(C; P_\bullet; \tau_\bullet)$ where $(C,\bullet) \in \Mgn[n]$ and where $\tau_i$ is a non zero $1$-jet of a formal coordinate at $P_i$.

Under certain assumptions (e.g., when the modules $W^i$ are simple and $V$ is $C_2$-cofinite), one can further forget the choice of the $1$-jets at the marked points and obtain a quasi-coherent sheaf on $\overline{\mathcal{M}}_{g,n}$, which we denote  $\VV_g(V;W^\bullet)$, or simply $\VV(V;W^\bullet)$, called the \emph{sheaf of coinvariants} (see \cite[\S8]{dgt2}).

\subsection{A look at the Chiral Lie algebra} \label{sec:chiral} We briefly recall here some key properties of the Lie algebra $\Ll_{C \setminus P_{\bullet}}(V)$, and we refer to \cite{dgt2} for a more detailed treatment. We first of all define $\mf{L}_{P_i}(V)=V \otimes \CC(\!(t_i)\!) / \im \nabla$ which acts on the module $W^i$ by $\left(v \otimes f(t_i)\right)(w) := \Res_{t_i=0}\left(f(t_i) Y(v,t_i)(w)\right)$ for every $v \in V$, $f(t_i) \in \CC(\!(t)\!)$ and $w \in W^i$. This induces the action of $\bigoplus_{i=0}^n\mf{L}_{P_i}(V)$ on the tensor product of $V$-modules $W^\bullet:= W^1 \otimes \dots \otimes W^n$.

 Morally, the Chiral Lie algebra $\Ll_{C \setminus P_{\bullet}}(V)$ consists of elements of $\bigoplus_{i=0}^n\mf{L}_{P_i}(V)$ that \textit{come from elements on $C\setminus P_{\bullet}$}. More precisely, when $C$ is smooth, \cite{bzf} define
\[ \Ll_{C \setminus P_{\bullet}}(V) = \text{H}^0(C\setminus P_{\bullet}, \mathcal{V}_C \otimes \Omega_C/\nabla_C),
\]where the sheaf $\mathcal{V}_C$ is a locally free sheaf of $\Oo_C$-modules associated with $V$, and $\nabla_C$ is a connection which is locally given by $\nabla$. Essential to this construction is the conformal structure on $V$ induced by $\omega$. Every element $\sigma \in \Ll_{C \setminus P_{\bullet}}(V)$ is mapped to its expansion $(\sigma_{P_i})_{i=1}^n \in  \bigoplus_{i=1}^n\text{H}^0(D_{P_i}^\times, \mathcal{V}_C \otimes \Omega_C/\nabla_C) \cong \bigoplus_{i=0}^n\mf{L}_{P_i}(V)$, where $D_{P_i}^\times$ is the punctured formal disk about $P_i$. Through this map, $\Ll_{C \setminus P_{\bullet}}(V)$ acts on $W^\bullet$, hence we can take coinvariants. From the gradation on $V$, one can lift any element of $\Ll_{C \setminus P_{\bullet}}(V)$ to an element of
\begin{equation} \label{eq:chiralsmooth}\bigoplus_{k \in \NN} V_k \otimes \text{H}^0\left(C \setminus P_{\bullet}, \Omega_C^{1-k} \right).
\end{equation}

When $C$ is nodal, a similar construction of $\Ll_{C \setminus P_{\bullet}}(V)$ can be given (see \cite[\S 3]{dgt2}), where $\Omega_C$ is replaced by $\omega_C$ and the sheaf $\mathcal{V}_C$ arises from $\mathcal{V}_{\widetilde{C}}$, where  $\widetilde{C}$ is the normalization of $C$. We give a more concrete realization of elements of $\Ll_{C \setminus P_{\bullet}}(V)$, similar to \eqref{eq:chiralsmooth},  used throughout. Assume that $C$ is a nodal curve with a single node $Q$ and let $\widetilde{C}$ be its normalization, with points $\Qp$ and $\Qm$ lying above $Q$. Then we can realize $\Ll_{C \setminus P_{\bullet}}(V)$ as the subquotient of \begin{equation} \label{eq:chiralnormaliz}\bigoplus_{k \in \NN} V_k \otimes \text{H}^0\left(\widetilde{C} \setminus P_{\bullet}, \Omega_{\widetilde{C}}^{1-k}\otimes \Oo_C(-(k-1)(\Qp+\Qm)) \right).
\end{equation} consisting of those elements $\sigma$ such that $[\sigma_{\Qp}]_0 = -\theta[\sigma_{\Qm}]_0$, where $\theta$ is defined as in \eqref{eq:theta} below. Note that for every element $\sigma$  of \eqref{eq:chiralnormaliz}, one has $\sigma_{\Qpm} \in \mf{L}_{\Qpm}(V)_{\leq 0}$.

\section{\texorpdfstring{$\UV^{\otimes 2}$}{U(V){2}}-modules from \texorpdfstring{$A(V)$}{\textbf{A}(V)}-bimodules}\label{sec:InducedModules}

Here we describe a functor $\Phi$ that takes an $A(V)$-bimodule to an  $\UV^{\otimes 2}$-module (see Definition\ref{def:ZFunDef}).
For instance, one obtains the $\UV^{\otimes 2}$-module $\Phi(A(V))$, which is later used to study vector space of coinvariants $[W^\bullet]_{\mathcal{L}_{C\setminus P_{\bullet}}(V)}$ at a stable pointed nodal curve $(C, P_{\bullet})$ and their factorization resolutions.

\subsection{The functor \texorpdfstring{$\Phi$}{\Phi}}\label{sec:Phi}
Given an $A(V)$-bimodule $E$, we show how to naturally associate a $\ms U(V)^{\otimes 2}$-module $\Phi(E)$. 
The surjective map $\ms U(V)_{0}\twoheadrightarrow A(V)$ from \cite[Proposition 3.1]{DongLiMason} induces an action of $\UV_0$ on $E$. Moreover, consider the involution $\theta$ on $\mathfrak{L}(V)_{0}$  defined by 
\begin{equation}
\label{eq:theta}
\theta\left(v_{[\deg(v)-1]}\right) := (-1)^{\deg(v)}\sum_{i\geq 0} \frac{1}{i!} \left(L_1^i v\right)_{[\deg(v) -i-1]}
\end{equation}
for homogeneous  $v \in V$. We note that $\theta$ is the same as in \cite[\S 4.1]{NT}, and has the opposite sign of the involution used in \cite{dgt2}, denoted $\vartheta$. 
\begin{defn}\label{def:ZFunDef} We define $\Phi \colon A(V)\text{-bimod} \to \UV^{\otimes 2}\text{-mod}$ as the functor which associates to every $A(V)$-bimodule $E$ the $\UV^{\otimes 2}$-module
\[ \Phi(E) = \rm{Ind}^{\UV^{\otimes 2}}_{\ms{U}(V)_{\le 0}^{\otimes 2}}E \cong \ms U(V)^{\otimes 2}\otimes_{\ms{U}(V)_{\le 0}^{\otimes 2}} E,
\]
where the action of $a\otimes b \in \ms{U}(V)_{\le 0}$ on $x \in E$ is given by
\[ (a \otimes b) \otimes x \mapsto \left\lbrace \begin{array}{ll}
 a   \cdot x \cdot \theta(b)    & \text{if } a,b \in \UV_{0}\\
 0 &\text{if } a,b \in \UV_{<0}
\end{array}\right.\]
By \cite[Proposition 4.1.1]{NT}, this is an action of bimodules since  $\theta(bb')=\theta(b')\theta(b)$ for all $b \in \UV_0$.
\end{defn}

\begin{remark}\label{rmk:Compatible}
By definition, $\Phi$ is compatible with the Verma module functor $\mathcal{M}$  from Eq~\eqref{eq:Verma}.
So in particular, if $E=I\otimes J^\vee$, for $A(V)$-modules $I$ and $J$, then $\Phi(E)=\mathcal{M}(I)\otimes \mathcal{M}(J)'$. By Lemma \ref{lem:Indecomposable}, if  $I$ and $J$
are indecomposable, then $\mathcal{M}(I)$ and  $\mathcal{M}(J)$ are indecomposable.  
\end{remark}
\subsection{Standard complex}\label{sec:FaRe} For $A=A(V)$, the  $A^{\otimes m}$ are $A$-bimodules, where the bimodule structure is induced from:
\begin{equation*}\label{eq:BigAction}\left(A \otimes A^{op} \right) \times A^{\otimes m} \to A^{\otimes m},  \ \ (a\otimes b, a_1 \otimes \cdots \otimes a_m)\mapsto  a \cdot a_1 \otimes a_2 \cdots \otimes a_{m-1} \otimes a_m \cdot  b.
\end{equation*}

\begin{defn}\label{def:lem1}The \textit{standard complex for $A$} is the exact sequence of $A$-bimodules
\begin{equation}\label{eq:complex}
\begin{tikzcd} \cdots \arrow{r}{\delta_3} & A^{\otimes 3}  \arrow{r}{\delta_2} &  A^{\otimes 2}  \arrow{r}{\delta_1} & A \arrow{r} & 0
\end{tikzcd},
\end{equation}
\begin{equation}\label{eq:deltad}\delta_{n+1}(a_0\otimes a_1 \otimes \ldots a_n \otimes a_{n+1})=\sum_{i=0}^n(-1)^ia_0 \otimes \cdots \otimes a_i a_{i+1} \otimes \cdots \otimes a_{n+1}.\end{equation}
In particular $\delta_2(a\otimes b \otimes c)=(ab\otimes c-a\otimes bc)$ and $\delta_1(a \otimes b)=a \cdot b$.
\end{defn}

Applying the functor $\Phi$ to the complex \eqref{eq:complex}, we obtain the complex of $\UV^{\otimes 2}$-modules
\begin{equation} \label{eq:complexA} \begin{tikzcd} \cdots \arrow{r}{\delta_3} & \Phi(A(V)^{\otimes 3})  \arrow{r}{\delta_2} & \Phi(A(V)^{\otimes 2}) \arrow{r}{\delta_1} & \Phi(A(V))\arrow{r} & 0
\end{tikzcd}.
\end{equation}
More explicitly
 \[\delta_2 \colon (u_1\otimes u_2)\otimes (a\otimes b \otimes c)\mapsto (u_1\otimes u_2) \otimes (ab\otimes c-a\otimes bc)\]
and if moreover $u_1 \otimes u_2 \in \UV_0$, one has
\[ \delta_2 \colon (u_1\otimes u_2)\otimes (a\otimes b \otimes c) \mapsto  u_1 a b \otimes c \theta(u_2) - u_1a \otimes bc \theta(u_2).
\]

\begin{remark} \label{rmk:ZAc} The complex \eqref{eq:complexA} is a natural extension of the definition and surjection of the $\UV^{\otimes 2}$-modules $Z \twoheadrightarrow \overline{Z}$ introduced in \cite[\S 6]{dgt2}.  To see that $\Phi(A(V)^{\otimes2})=Z$, and that $\overline{Z}=\Phi(A(V))$,  recall that $Z$ is defined as $\left(\UV_{\UV_{\leq 0}}A(V)\right)^{\otimes 2}$, where $\UV_{\leq 0}$ acts on $A(V)$ via the usual projection $\LV_{\leq 0} \to A(V)$. It follows that the involution of $A(V)\otimes A(V)$ given by $a \otimes b \mapsto a \otimes \theta(b)$ induces the isomorphism of $\UV^{\otimes 2}$-modules $Z \to \Phi(A(V)^{\otimes2})$ given by \[(u_1 \otimes a_1) \otimes (u_2 \otimes a_2) \mapsto (u_1 \otimes u_v \otimes a_1 \otimes \theta(a_2))\] for all $u_i \in \UV$ and $a_i \in A(V)$.
\end{remark}

\subsection{Coinvariants} Let $\Ct$ be a smooth curve, marked by disjoint points $P_1$, $\dots$, $P_n$, with formal coordinates $t_1$, $\dots$, $t_n$. Let $W^1, \dots, W^n$ be $V$-modules and denote by $W^\bullet$ their tensor product $W^1 \otimes \cdots \otimes W^n$. Fix two more points $\Qp$ and $\Qm$, distinct from eachother and disjoint from $P_{\bullet}$, together with formal coordinates $s_{\pm}$ at $\Qpm$, respectively. 
We can then define the functor
\begin{equation} \label{eq:functorcoinv} [W^\bullet \otimes - \, ]_{\LCtPQ} \colon {\UV^{\otimes 2}{\textbf{-mod}}} \longrightarrow {k{\textbf{-mod}}},  \qquad \mathcal{X} \mapsto [W^\bullet \otimes \mathcal{X} ]_{\LCtPQ}.
\end{equation}

\begin{defn}\label{def:FunctorV} We denote by $\FV=\FV_{\LCtPQ}(W^\bullet)$ the right-exact functor from the category of $A(V)$-bimodules, to the category of $k$-modules given by the composition $\Phi$ and of \eqref{eq:functorcoinv}. 
\end{defn}

\begin{lem} \label{lem:lem2}The map induced applying the functor \eqref{eq:functorcoinv} to $\delta_2(\Phi(A(V)^{\otimes2}))$ (from \eqref{eq:deltad}) factors as
\[\xymatrix @C=0.4cm {
[W^\bullet \otimes \Ac^{3}]_{\Lc_{\Ct \setminus P_{\bullet} \cup \Qpm}(V)} \ar@{->>}[drr]_{\phi}\ar[rrrr]^{[id\otimes \delta_2]}  &&&& [W^\bullet \otimes \Phi(A(V)^{\otimes2})]_{\Lc_{\Ct \setminus P_{\bullet} \cup \Qpm}(V)} \\
&&[W^\bullet \otimes \ker(\delta_1)]_{\Lc_{\Ct \setminus P_{\bullet} \cup \Qpm}(V)}  \ar[urr].
}\]
\end{lem}
\begin{proof}
To the exact sequence \eqref{eq:complex} from Lemma \ref{def:lem1}, we apply the functor $\FV$ defined above. Since it is right-exact, the lemma holds.
\end{proof}

\section{Factorization presentations} \label{sec:main}
 
Here we prove our main results regarding factorizaton presentations of nodal coinvariants, obtained from  resolutions of $A(V)$  by factorizable bimodules, described in Definition \ref{def:FR}. For this section we let $(C,P_{\bullet},t_\bullet)$ be a stable pointed coordinatized curve with only one node $Q$ and such that $C \setminus P_\bullet$ is affine. Denote by $\eta \colon \widetilde{C}\to C$ the normalization of $C$,  $\Qpm=\eta^{-1}Q$ with coordinates $s_\pm$, and for $1\le i \le n$, let $W^i$ be a $V$-module.

\begin{defn}\label{def:FR}
Let $A$ be an associative algebra, and $E$ an $A$-bimodule.  We say that
\begin{enumerate}
    \item 
$E$ is a {\em{factorizable bimodule}} if can be written as a finite sum 
\[E=  \bigoplus \left(X_0 \otimes Y_0 \right),\]
of tensor products of left $A$-modules $X_0$, with right $A$-modules $Y_0$.
\item A {\em{factorization resolution}} of  $E$ is a resolution by factorizable $A$-bimodules
\[\xymatrix{\cdots \ar[r]^-\alpha & \bigoplus \left(X_0 \otimes Y_0 \right) \ar[r] & E \ar[r] &  0},\]
where all maps are bimodule morphisms.
\end{enumerate}
\end{defn}

\begin{thm}\label{thm:FactorizationRes} Suppose $V$ is a  $C_1$-cofinite vertex operator algebra of CFT-type.  Then 
\begin{enumerate}
\item  $A(V)$ has a factorization resolution as in Definition \ref{def:FR}, part (b); 
\item To any factorization resolution $\cdots  \overset{\alpha}{\rightarrow} \oplus \left(X_0 \otimes Y_0 \right) \rightarrow A(V) \rightarrow 0$ of $A(V)$, considered as a bimodule over itself, one has
 \begin{equation}\label{eq:FRes}
[W^\bullet]_{\LCP}\cong 
\dfrac{ \bigoplus [W^{\bullet}\otimes \Phi(X_0\otimes Y_0)]_{\LCtPQ}}{{\rm{Image}}(\FV(\alpha))},
\end{equation}
We refer to \eqref{eq:FRes} as a factorization presentation of $[W^{\bullet}]_{\LCP}$.
\item If $V$ is $C_2$-cofinite, there is a  factorization presentation of $[W^{\bullet}]_{\LCP}$ indexed by a finite sum of (isomorphism classes) of indecomposable $V$-modules.
\end{enumerate}
\end{thm}

To prove Theorem \ref{thm:FactorizationRes}, we use Proposition \ref{prop:AbstractFactorMain}, the main technical result here.

\begin{prop}\label{prop:AbstractFactorMain} For a  $C_1$-cofinite vertex operator algebra $V$ of CFT type, the map \[W^\bullet \to W^\bullet \otimes \Phi(A(V)), \qquad w \mapsto w \otimes 1\] induces an isomorphism
\begin{equation}
[W^\bullet]_{\LCP}\cong
[W^\bullet \otimes \Phi(A(V))]_{\LCtPQ}.
\end{equation}
\end{prop}

\proof
We consider the standard bimodule factorization resolution of $A(V)$ as described in Definition \ref{def:lem1}. We may apply $\Phi$ to $A(V)^{\otimes m}$, to obtain the $\UV^{\otimes 2}$-module $\Phi(A(V)^{\otimes m})$, denoted here $\Ac^m$ for simplicity.  In accordance with \cite{dgt2}, we denote by $\mathcal{L}_{\widetilde{C}\setminus P_{\bullet}}(V,\{Q_+,Q_-\})$ the Lie subalgebra of $\mathcal{L}_{C\setminus P_{\bullet}}(V)$ consisting of those elements $\sigma$ with $\deg \sigma_{\Qpm} < 0$; which gives the top row of the following diagram:

\begin{equation*} \xymatrix{
\ker(\pi) \ar[d]^{h_\theta}_\cong \ar@{^(->}[r] &  \left[ W^\bullet\right]_{\mathcal{L}_{\widetilde{C}\setminus P_{\bullet}} (V,\{Q_+,Q_-\})} \ar[r] \ar[d]^{h_\theta}_\cong \ar@{->>}[r]^{\pi} & \left [ W^\bullet \right ]_{\LCP} \ar[d]^{h_\theta}_\cong \\
{h_\theta}(\ker(\pi)) \ar@{^(->}[r] &  \left[ W^\bullet\otimes \Ac^2 \right]_{\LCtPQ} \ar@{->>}[r] \ar@{=}[d] & \left[ W^\bullet\otimes \Ac^2 \right]_{\LCtPQ} / {h_\theta}(\ker(\pi)) \\
\left[ W^\bullet\otimes \Ac^3 \right]_{\LCtPQ} \ar[r]^{[id\otimes \delta_2]} \ar@{->>}[d]_\phi &  \left[ W^\bullet\otimes \Ac^2 \right]_{\LCtPQ} \ar@{->>}[r]^{[id\otimes \delta_1]} & \left[ W^\bullet\otimes \Ac \right]_{\LCtPQ} \\
\left[ W^\bullet \otimes \ker(\delta_1)\right]_{\LCtPQ} \ar@{->}[ur]
}
\end{equation*}

 The middle row is induced from the top one via the isomorphism $h$ established in \cite[Proposition 6.2.1]{dgt2},  induced via the map $W^\bullet \to W^\bullet \otimes A(V)\otimes A(V)$ given by $w \mapsto w \otimes 1\otimes 1$, and the isomorphism of $\Ac^2$ with $Z$ described in Remark \ref{rmk:ZAc}. The factorization of $[id \otimes \delta_2]$ was established in Lemma \ref{lem:lem2}. To conclude it suffices to show that  ${h_\theta}(\ker(\pi))$ is the image of $W^\bullet \otimes \ker(\delta_1)$ in $\left[ W^\bullet\otimes \Ac^2 \right]_{\LCtPQ}$. For, we proceed in Steps $(A)$, $(B)$, and $(C)$.\\

\noindent
\textbf{Step (A)} Description of ${h_\theta}(\ker(\pi)))$ (see also \cite[Proof of Theorem 7.0.1]{dgt2}). Note that $\ker(\pi)$ is generated by elements of the form $[\sigma(w)]$ for $\sigma \in  \Lc_{C\setminus P_{\bullet}}(V)$ and $w \in W^\bullet$. Via the map $h_\theta$, such an element is sent to $[\sigma(w) \otimes 1 \otimes 1]$ which is equivalent to the element $[-w \otimes \sigma_{\Qp}(1) \otimes 1 -w \otimes 1 \otimes \sigma_{\Qm}(1)]$. Note moreover that since $\sigma \in \Lc_{C\setminus P_{\bullet}}(V)$ and $1$ is annihilated by $\UV_{>0}$, one has that $\sigma_\pm(1) \in \UV_0$ and further that $\theta(\sigma_+(1))=-\sigma_-(1)$. Consequently, we can deduce that $h_\theta(\ker(\pi))$ is generated by elements of the form $[w \otimes a \otimes 1 - w \otimes 1 \otimes \theta(a)]$ for $w \in W^\bullet$ and $a \in A$. Applying the isomorphism between $Z$ and $\Ac^2$ described in Remark \ref{rmk:ZAc} this implies that $h_\theta(\ker(\pi))$ is generated by elements of the form
\begin{equation}\label{eq:hkerpi}[ w \otimes a \otimes 1 - w \otimes 1 \otimes a]
\end{equation} for $w \in W^\bullet$ and $a \in A$.\\

\noindent
\textbf{Step (B)} Description of $\ker([id \otimes \delta_1])$. We claim $K:=\ker(\delta_1)$ consists of elements 
\[ (u_1 \otimes u_2) \otimes (a \otimes 1) - (u_1 \otimes u_2) \otimes (1 \otimes a)\]
for $u_i \in \UV$ and $a \in A(V)$, hence the image of $W^\bullet \otimes K$ in $\left[ W^\bullet\otimes \Ac^2 \right]_{\LCtPQ}$ is generated by elements of the form
\begin{equation}\label{eq:keriddelta} [w \otimes (u_1 \otimes u_2) \otimes (a \otimes 1) - w\otimes (u_1 \otimes u_2) \otimes (1 \otimes a)]
\end{equation}
for $w \in W^\bullet$, $u_i \in \UV$ and $a \in A(V)$.\\

\noindent
\textbf{Step (C)} Comparison of \eqref{eq:hkerpi} and \eqref{eq:keriddelta}. Note that, by setting $u_1 = u_2 = 1$, we deduce that $h_\theta(\ker(\pi))$ is contained in $\ker([id \otimes \delta_1])$. We are left to show that the opposite inclusion holds.

Write $u_i = u_i^{\ell_i} \dots u_i^{1} \otimes 1$ for $\ell_i \in \NN$ and $u_i^{*} \in \LV$ and call $\ell_i$ the length of $u_i$. We will show the following: if all elements as in \eqref{eq:keriddelta} having $(u_1, u_2)$ of length at most $(\ell_1, \ell_2)$ belong to $h_\theta(\ker(\pi))$, then all elements as in \eqref{eq:keriddelta} having $(u_1, u_2)$ of length at most $(\ell_1+1, \ell_2)$ (and $(\ell_1, \ell_2+1)$) belong to $h_\theta(\ker(\pi))$ too. We then conclude by induction, having proved the base case $\ell_1=\ell_2=0$ in part (B).

Assume for all $w \in W^\bullet$, $a \in A(V)$  and $u_i$ of length $\ell_i$, that:
\begin{equation}\label{eq:newIH} w \otimes (u_1 \otimes u_2) \otimes (a \otimes 1) - w \otimes (u^1 \otimes u_2) \otimes (1 \otimes a) \in h_\theta(\ker(\pi)).
\end{equation}
By the symmetric roles played by $u_1$ and $u_2$, it is enough to show, for all $u_1^{\ell+1} \in \LV$,  the assumption of \eqref{eq:newIH} implies that
\begin{equation}\label{eq:induction} w \otimes (u_1^{\ell+1}u_1 \otimes u_2) \otimes (a \otimes 1) - w \otimes (u_1^{\ell+1} u^1 \otimes u_2) \otimes (1 \otimes a) \in h_\theta(\ker(\pi)).
\end{equation}
As an application of Riemann-Roch, we may choose $\tau \in \Lc_{\Ct \setminus P_{\bullet} \cup Q\pm}(V)$ so that
\[ \tau_{\Qp} u_1 =  u_1^{\ell+1} u_1 \qquad \text{ and } \qquad \tau_{\Qm} u_2=0\] hold true in $\UV/{\UV_{<0}}$.
It follows that in $[W^\bullet \otimes \Ac^2]_{\LCtPQ}$ one has
\begin{align*} 0 &= \tau\left( w \otimes (u_1 \otimes u_2) \otimes (a \otimes 1) - w \otimes (u_1 \otimes u_2) \otimes (1 \otimes a \right)\\
&= \tau_{P_{\bullet}}(w) \otimes (u_1 \otimes u_2) \otimes (a \otimes 1) - \tau_{P_{\bullet}}(w) \otimes (u_1 \otimes u_2) \otimes (1 \otimes a)\\
& \qquad  + w \otimes (\tau_{\Qp} u_1 \otimes u_2) \otimes (a \otimes 1) - w \otimes (\tau_{\Qp} u_1 \otimes u_2) \otimes (1 \otimes a)\\
&\qquad + w \otimes (u_1^{\ell_1} \otimes \tau_{\Qm}u_2) \otimes (a \otimes 1) - w \otimes (\tau_{\Qp} u_1 \otimes \tau_{\Qm}u_2) \otimes (1 \otimes a)\\
&= \tau_{P_{\bullet}}(w) \otimes (u_1 \otimes u_2) \otimes (1 \otimes 1) - \tau_{P_{\bullet}}(w) \otimes (u_1 \otimes u_2) \otimes (1 \otimes a)\\
&\qquad + w \otimes  (u_1^{\ell+1} u_1 \otimes u_2) \otimes (a \otimes 1) - w \otimes ( u_1^{\ell+1} u_1 \otimes u_2) \otimes (1 \otimes a).
\end{align*}
Hence one has
\[w \otimes  (u_1^{\ell+1} u_1 \otimes u_2) \otimes (a \otimes 1) - w \otimes ( u_1^{\ell+1} u_1 \otimes u_2) \otimes (1 \otimes a)\]
equals
\[ -\tau_{P_{\bullet}}(w) \otimes (u_1 \otimes u_2) \otimes (a \otimes 1) + \tau_{P_{\bullet}}(w) \otimes (u_1 \otimes u_2) \otimes (1 \otimes a)\]
in  $\left[ W^\bullet\otimes \Ac^{2} \right ]_{\LCtPQ}$. By the induction hypothesis, this latter element is in $h_\theta(\ker(\pi))$, hence \eqref{eq:induction} holds true, concluding the argument. \endproof

\begin{remark} \label{rmk:morenodes}
Similarly, one can shows the following. Let $(C, P_\bullet, t_\bullet)$ be a stable coordinatized  curve with exactly $\delta$ nodes $Q_1, \dots, Q_\delta$ and such that $C \setminus P_\bullet$ is affine. Denote by $(\Ct,P_\bullet \sqcup Q_{\star,\pm}, t_\bullet\sqcup s_{\star,\pm})$ its normalization: this is a, possibly disconnected, smooth coordinatized $(n+2\delta)$-pointed curve. Let $V$ be a $C_1$-cofinite VOA, and let $W^1, \ldots, W^n$ be $V$-modules. Then we have
\begin{equation} \label{eq:prop33morenodes} \left[W^\bullet\right]_{(C,P_\bullet)} \, \cong \, \left[ W^\bullet \otimes \Phi(A(V))^\star \right]_{(\Ct, P_\bullet \cup Q_{\star,\pm})},
\end{equation} where $\Phi(A(V))^\star := \Phi(A(V)) \otimes \dots \otimes \Phi(A(V))$.
\end{remark}

After proving Lemma \ref{lem:BimodDecomp} (a well known result from ring theory that we provide for convenience), we prove Theorem \ref{thm:FactorizationRes}.

\begin{lem}\label{ringdecomp} \label{lem:BimodDecomp}
Every finite dimensional algebra $A$ can be uniquely written as a product $A = A_1 \oplus \cdots \oplus A_m$ of indecomposable bimodules (which are ideals). In particular, if $V$ is $C_2$-cofinite, then $A(V)$ has a unique bimodule decomposition as a sum of indecomposable bimodules.
\end{lem}
\begin{proof}
Either $A$ is itself indecoposable as a bimodule, in which case we are done, or $A$ can be written as a product $A' \oplus A''$ of ideals. By induction on the dimension of $A$ (the case of dimension $1$ being trivial as it must be indecomposable in this case), $A$ can be written as such a product.

Writing $1 = e_1 + \cdots + e_m$ with $e_i \in A_i$ it follows that $A_i = e_i A$ and the elements $e_i$ are central, pairwise orthogonal idempotents: we have $1^2 = \sum_{i,j} e_i e_j$ but $1^2 = 1 = \sum_i e_i$. But as we have a product decomposition of ideals, we find $e_ie_j \in A_i A_j \subset A_i \cap A_j = 0$ if $i \neq j$. Therefore $1^2 = \sum e_i^2 = \sum e_i = 1$ implies (by the product decomposition) $e_i^2 = e_i$. Hence the elements $e_i$ are pairwise orthogonal idempotents. Further, for $a \in A$, we have $1a = a1$ so $\sum a e_i = \sum e_i a$. Since $e_i a, a e_i \in A_i$ again the product decomposition imples $a e_i = e_i a$ for each $i$, showing that the $e_i$ are central.

For uniqueness, suppose we have two such decompositions $A = A_1 \oplus \cdots \oplus A_m = B_1 \oplus B_\ell$. In particular, we find $B_j = 1 B_j = e_1 B_j + e_2 B_j + \cdots e_m B_j$ with $e_p B_j \cap e_q B_j \subset A_p \cap A_q$ which is $0$ if $p \neq q$. But as $B_j$ is indecomposable, it follows that $B_j \subset A_i$ for some $i$. But reversing the roles of the $A_i$ and $B_j$ we find that the decompositions must agree, showing uniqueness.
\end{proof}

\subsection{Proof of Theorem \ref{thm:FactorizationRes}} \label{sec:FactResProof}
The first claim follows from Lemma  \ref{lem:BimodDecomp}.

We now prove the second claim. As described in  \S   \ref{sec:InducedModules}, the standard complex gives a factorization resolution of $A(V)$.  Applying the functor $\FV$ from Definition \ref{def:FunctorV} to any factorization resolution of $A(V)$, $\xymatrix{\cdots  \ar[r]^-{\alpha_2} & \bigoplus \left(X_{0}\otimes Y_0 \right) \ar[r]^-{\alpha_1} & A(V) \ar[r] & 0}$, we obtain a right exact sequence
\[\xymatrix{
\cdots \ar[r]^-{\FV(\alpha_2)} & \left[W^{\bullet} \otimes \Phi \left(\oplus \left(X_{0}\otimes Y_0 \right)\right)\right]_{\LCtPQ}
\ar[r]^-{\FV(\alpha_1)} & \left[ W^\bullet \otimes \Phi(A(V))\right]_{\LCtPQ} \ar[r] & 0}.
\]
  So by Proposition~\ref{prop:AbstractFactorMain}, we conclude that
 \begin{multline}
 [W^{\bullet}]_{\LCP}\cong [W^{\bullet} \otimes \Phi(A(V))]_{(\widetilde{C},P_{\bullet}\cup \Qpm)} \\
 \cong {\rm{Coker}}(\FV(\alpha_1))
 \cong \bigoplus_{}  \bigslant{[W^{\bullet} \otimes \Phi \left(X_{0}\otimes Y_0 \right)]_{\LCtPQ}}{{\rm{Image}}(\FV(\alpha_1))}.\end{multline}

For the third claim, apply  $\Phi$ to the bimodule resolution given by the standard decomposition of each bimodule in the bimodule decomposition of $A(V)$ from Lemma \ref{lem:BimodDecomp}. 
The claim follows from compatibility of $\Phi$ with the functors $\mathcal{M}$ (Remark \ref{rmk:Compatible}) and by Lemma \ref{lem:Indecomposable}, giving that $\mathcal{M}$ takes indecomposable $A(V)$ modules to indecomposable $V$-modules.  \qed

\section{Consequences}\label{sec:Consequences}
Here we show that there is a simpler expression for nodal coinvariants defined by representations of $C_2$-cofinite VOAs, and such vector spaces are finite dimensional. Let $V$ be a $C_2$-cofinite VOA so that, in view of Lemma \ref{lem:BimodDecomp}, there is a bimodule decomposition 
\begin{equation}\label{eq:ADecomp}A(V)=\bigoplus_{I \in \mathscr{I}} I,\end{equation}
as a sum of indecomposable $A(V)$-bimodules. In Corollary \ref{cor:FactorSum}, we will refer to the sets
\[\mathscr{S}:=\{S_0 \otimes S_0^\vee  \in \mathscr{I} \ : \  S_0 \text{ is simple} \} \ \mbox{ and } \
\SV:=\{S_0 \otimes S_0^\vee  \in \mathscr{S} \  : \ \mathcal{M}(S_0)=L(S_0)\},\]
were $\mathcal{M}$ and $L$ denote Zhu's functors (see \S \ref{sec:background}).  The inclusions $\mathscr{S} \mathscr{V} \subseteq \mathscr{S} \subseteq \mathscr{I}$ can be equalities, as they are when $V$ is rational (see Remark \ref{rmk:Indecomposable}). 

\begin{cor}\label{cor:FactorSum} If $V$ is $C_2$-cofinite, then $[W^\bullet]_{\LCP}$ is isomorphic to 
\[ 
\bigoplus_{\stackrel{S_0 \otimes S_0^\vee \in }{ \SV}} [W^\bullet \otimes S \otimes S']_{(\widetilde{C},P_{\bullet}\cup \Qpm)}
\, \oplus \, \bigoplus_{\stackrel{S_0 \otimes S_0^\vee \in }{\mathscr{S} \setminus \SV}}  [W^\bullet \otimes \mathcal{M}(S_0) \otimes \mathcal{M}(S_0^\vee) ]_{(\widetilde{C},P_{\bullet}\cup \Qpm)} 
\, \oplus \, \bigoplus_{\stackrel{I \in }{\mathscr{I}\setminus \mathscr{S}}} [W^\bullet \otimes \Phi( I)]_{(\widetilde{C},P_{\bullet}\cup \Qpm)}.\] \end{cor}
\begin{proof}
The result follows from Proposition \ref{prop:AbstractFactorMain} and that  taking coinvariants commutes with taking sums. More specifically, we can apply Proposition \ref{prop:AbstractFactorMain} to each summand of the unique bimodule decomposition of $A(V)$ given by Lemma \ref{lem:BimodDecomp}.
\end{proof}

\begin{cor}\label{cor:CoherentMain} If $V$ is $C_2$-cofinite, then $\mathbb{V}^{J}(V,W^\bullet)$ is coherent on $\overline{\mathcal{J}}^\times_{g,n}$. In particular, when it descends to $\bMgn[n]$, the sheaf of coinvariants $\mathbb{V}(V,W^\bullet)$ is coherent.
\end{cor}

\begin{remark}
By \cite{dgt2}, if $V$ is $C_2$-cofinite, and $V$-modules $W^i$ are simple, then $\mathbb{V}^{J}(V,W^\bullet)$ descends to $\mathbb{V}(V,W^\bullet)$ on $\bMgn[n]$.
\end{remark}

\begin{proof} Coherence is a local property, hence it suffices to show that spaces of coinvariants are finite dimensional for every stable coordinatized curve. Since by \cite[Theorem 5.1.1]{dgt2}, coinvariants at smooth pointed curves are finite dimensional, it is enough to show nodal coinvariants are finite dimensional. 

Consider a coordinatized curve
$(C,P_{\bullet}, t_\bullet)$ with exactly $\delta$ nodes $Q_1, \dots, Q_\delta$. By propagation of vacua \cite{codogni}, we can assume that $C \setminus P_\bullet$ is affine.  Let $\eta \colon \widetilde{C}\rightarrow C$ be the normalization of $C$ and let $Q_{i,\pm}=\eta^{-1}(Q_i)$, so that $(\Ct, P_{\bullet} \sqcup Q_{\star,\pm}, t_\bullet \sqcup s_{\star, \pm})$ is a smooth coordinatized curve. Applying \eqref{eq:prop33morenodes} we deduce that 
$[W^\bullet]_{\LCP}$ is isomorphic to $\left[ W^\bullet \otimes \Phi(A(V))^* \right]_{(\Ct, P_\bullet \cup Q_{\star,\pm})}$.  Using the standard bimodule resolution $\xymatrix{\cdots \ar[r] & A(V) \otimes A(V) \ar[r] & A(V) \ar[r] & 0}$, by compatibility of $\Phi$ with the Verma module functor $\mathcal{M}$, the dimensions of the vector space of coinvariants $[W^\bullet]_{\LCP}$ is bounded above by dimension of the vector space  \begin{equation}\label{eq:bound}\left[W^\bullet \otimes (\mathcal{M}(A(V)) \otimes \mathcal{M}(A(V)))^{\otimes \delta}\right]_{\LCtPQ}.\end{equation} Since $W^i$ and $\mathcal{M}(A(V))$ are finitely generated, and $\Ct$ is smooth, by \cite[Theorem 5.1.1]{dgt2}, the space in \eqref{eq:bound} is finite dimensional, and hence so is $[W^\bullet]_{\LCP}$.
\end{proof}

\section{Examples}\label{sec:Examples}
Here we apply our results to two well-known families that are $C_2$-cofinite and not rational: the triplet algebras and the Symplectic Fermions. 

\subsection{The Triplet algebras \texorpdfstring{$\mathcal{W}(p)$}{W(p)}}\label{ex:Triplet} We consider a family of non-rational, $C_2$-cofinite VOAs, strongly finitely generated by $1$, $\omega$, and three elements in degree $2p-1$, for $p\in \ZZ_{\geq 2}$.  
There are $2p$ non-isomorphic simple $\mathcal{W}(p)$-modules called by different notation in the literature including $\{\Lambda(i), \Pi(i)\}_{i=1}^p$ in \cite{AM} (and $\{X_s^{\pm}: 1 \le s \le p\}$ in eg.~\cite{NTTrip, TW}).  The corresponding simple $A(\mathcal{W}(p))$-modules are denoted $\{\Lambda(i)_0, \Pi(i)_0\}_{i=1}^p$ in \cite{AM} (and by $\{\overline{X}_s^{\pm}: 1 \le s \le p\}$ in \cite{NTTrip, TW}).  By \cite{AM, NTTrip}, there is a bimodule decomposition 
\begin{equation}\label{eq:BD}
A(\mathcal{W}(p))\cong  \bigoplus_{i=1}^{2p} B_i,
\end{equation}
where components $B_i$ are described as follows: 
\begin{itemize}
\item For $1\le i \le p-1$, we have that  $B_i  \cong \mathbb{C}[\epsilon]/\epsilon^2 \cong \mathbb{I}_{h_i,1}$, which is indecomposable and reducible. In \cite{NTTrip,TW} this is denoted by $\widetilde{\overline{X}}_i^+$ and it is the projective cover of the simple $A(\mathcal{W}(p))$-module $\overline{{X}}_i^+$.  
\item $B_{p} \cong \mathbb{C}\cong \Lambda(p)_0 \otimes \Lambda(p)_0^\vee$.
\item For $1  \le i \le p$, we have that $B_{p+i} \cong M_{2}(\mathbb{C}) \cong \Pi(i)_0 \otimes \Pi(i)_0^\vee$. 
 \end{itemize}
  
\begin{remark}\label{rmk:Milas} One can see that the generalized Verma modules induced from the irreducible indecomposable $A(\mathcal{W}(p))$-modules $\Lambda(p)_0$, and $\Pi(p)_0$, are simple (see  \cite[page 2678]{AM}).  In particular,   $\Lambda(p)=\mathcal{M}(\Lambda(p)_0)=L(\Lambda(p)_0)$ and $\Pi(p)=L(\mathcal{M}(\Pi(p)_0)=(\Pi(p)_0))$ \cite{AMSelecta}.
 \end{remark}
 
\noindent
Consider now a stable pointed curve $(C,P_{\bullet})$ as in \S\ref{sec:main}; that is the curve $C$ has one node $Q$, we denote by $\eta \colon \widetilde{C}\to C$ its normalization, and $\Qpm=\eta^{-1}(Q)$. In order to give an application of Corollary \ref{cor:FactorSum}, we identify $\SV$ more explicitly. From Remark \ref{rmk:Milas}, the modules $\Lambda(p)_0 \otimes \Lambda(p)_0^\vee$ and $\Pi(p)_0 \otimes \Pi(p)_0^\vee$ are elements of $\SV$. It follows that $\SV \setminus \{ \Lambda(p)_0 \otimes \Lambda(p)_0^\vee, \Pi(p)_0 \otimes \Pi(p)_0^\vee\}$ equals $\mathcal{S}:= \{1 \le i \le p-1  :  \mathcal{M}(\Pi(i)_0)=L(\Pi(i)_0)\}$. By Corollary \ref{cor:FactorSum}
\begin{multline}\label{eq:TripRes}
 [W^\bullet]_{\LCP} \cong [W^\bullet \otimes \Lambda(p) \otimes \Lambda(p)']_{(\widetilde{C},P_{\bullet}\cup \Qpm)} \oplus \bigoplus_{i \in  \mathcal{S}\cup\{p\}}[W^\bullet \otimes \Pi(i)  \otimes \Pi(i)']_{(\widetilde{C},P_{\bullet}\cup \Qpm)} \\
\oplus \bigoplus_{\stackrel{1\le i \le p-1}{i \notin \mathcal{S}}} [W^\bullet \otimes \mathcal{M}(\Pi(i)_0)\otimes \mathcal{M}((\Pi(i)_0)^\vee)]_{(\widetilde{C},P_{\bullet}\cup \Qpm)} 
\oplus \bigoplus_{i=1}^{p-1} [W^\bullet \otimes \Phi(\mathbb{I}_{h_i,1})]_{(\widetilde{C},P_{\bullet}\cup \Qpm)}.\end{multline} 
Moreover, one can refine the last $p-1$ summands on the right hand side of \eqref{eq:TripRes}:

\begin{prop}\label{prop:Simplification} There is an isomorphism \[[W^\bullet \otimes \Phi(\mathbb{I}_{h_i,1})]_{(\widetilde{C},P_{\bullet}\cup \Qpm)}\cong \bigslant{[W^\bullet \otimes \mathcal{M}(\mathbb{I}_{h_i,1}) \otimes \mathcal{M}(\mathbb{I}_{h_i,1})']_{(\widetilde{C},P_{\bullet}\cup \Qpm)}}{{\rm{Image}}(r_\epsilon + \ell_\epsilon)},\]
where $r_\epsilon$, and $\ell_\epsilon$ are square zero linear endomorphisms, acting on the right and left factors in the tensor product $\mathcal{M}(\mathbb{I}_{h_i,1}) \otimes \mathcal{M}(\mathbb{I}_{h_i,1})'$ respectively.
\end{prop}

To more explicitly define $r_\epsilon$ and $\ell_{\epsilon}$, and prove the claim, we first establish Lemma \ref{lem:IDecomposition}. For this, we consider the ring $\mathbb{I}=\mathbb{I}_{h_i,1}\cong \mathbb{C}[\epsilon]/\epsilon^2$, as a bimodule
$\mathbb{I}^+=\mathbb{C} 1 \oplus \mathbb{C} \epsilon$ over itself,
with action  (ambiguously) denoted $1\cdot \epsilon = \epsilon \cdot 1=\epsilon$  and  $\epsilon \cdot \epsilon = 0$.

\begin{lem}\label{lem:IDecomposition} The indecomposable and not simple (bi)modules 
$\mathbb{I}^+$
have periodic resolutions
\begin{equation}\label{eq:Res}\xymatrix{\cdots  \ar[r]^-g & \mathbb{I}^+\otimes \mathbb{I}^+ \ar[r]^-f & \mathbb{I}^+\otimes \mathbb{I}^+ \ar[r]^-g & \mathbb{I}^+\otimes \mathbb{I}^+ \ar[r]^-m & \mathbb{I}^+ \ar[r] & 0},
\end{equation}
where $g(1\otimes 1)=\epsilon \otimes 1 + 1 \otimes \epsilon$, $f(1\otimes 1)=\epsilon \otimes 1 - 1 \otimes \epsilon$, and $m(1\otimes 1)=1$.
\end{lem}
\proof We will refer to the auxiliary indecomposable $\mathbb{I}$-module
$\mathbb{I}^-=\mathbb{C} \eta \oplus \mathbb{C} \delta$, 
with action 
\[ \epsilon \cdot \delta = \delta \cdot \epsilon=0, \ \mbox{ and }  \ \epsilon \cdot \eta = \delta = -\eta \cdot \epsilon. \]
The result is obtained by splicing together the following  two short exact sequences: First
\[\xymatrix{0 \ar[r] & \mathbb{I}^- \ar[r]^-h &  \mathbb{I}^+\otimes \mathbb{I}^+ \ar[r]^-m & \mathbb{I}^+\ar[r] &  0},\]
where $h(\eta)=\epsilon \otimes 1 - 1\otimes \epsilon$,   $-h(\delta)=\epsilon \otimes \epsilon$, and $m(1\otimes 1)=1$.  The second
\[\xymatrix{0 \ar[r] & \mathbb{I}^+ \ar[r]^-g &  \mathbb{I}^+\otimes \mathbb{I}^+ \ar[r]^-{\overline{m}} & \mathbb{I}^- \ar[r] &  0},\]
with $g(1)=\epsilon \otimes 1 +1\otimes \epsilon$,  $g(\epsilon)=\epsilon \otimes \epsilon$,  with kernel generated by $\epsilon \otimes 1 + 1\otimes \epsilon$, and 
such that $\overline{m}(1\otimes 1)=\eta$,  $\overline{m}(\epsilon \otimes 1)= \epsilon \cdot \eta = \delta$,
$\overline{m}(1 \otimes \epsilon)=\eta \cdot \epsilon = - \delta$, and $\overline{m}(\epsilon \otimes \epsilon)=\epsilon \cdot \eta \cdot \epsilon = 0$. \endproof

\proof[Proof of Proposition \ref{prop:Simplification}]
From \eqref{eq:Res} we can provide another description of the term $[W^\bullet \otimes \Phi(\mathbb{I}_{h_i,1})]_{(\widetilde{C},P_{\bullet}\cup \Qpm)}$.  
 Note that, since $A(\mathcal{W}(p))$ is commutative, right and left multiplication by $\epsilon$, considered as an element of $B_i = \mathbb{I}_{h_i,1} \subset A(\mathcal{W}(p))$ induce $A(\mathcal{W}(p))$-bimodule morphisms $\ell_\epsilon, r_\epsilon: \mathbb{I}_{h_i,1} \rightarrow \mathbb{I}_{h_i,1}$. 
 Consequently, these also induce linear endomorphisms $r_\epsilon, \ell_\epsilon$ of $[W^\bullet \otimes \mathcal{M}(\mathbb{I}_{h_i,1}) \otimes \mathcal{M}(\mathbb{I}_{h_i,1})']_{(\widetilde{C},P_{\bullet}\cup \Qpm)}$, acting on the right and left factors in the tensor product $\mathbb{I}_{h_i,1} \otimes \mathbb{I}_{h_i,1}^\vee$ respectively. These endomorphisms are non-identity, and in fact, square $0$ (as multiplication by $\epsilon$ is square $0$). The right exactness of  \eqref{eq:Res} then tells us that we can identify
$[W^\bullet \otimes \Phi(\mathbb{I}_{h_i,1})]_{(\widetilde{C},P_{\bullet}\cup \Qpm)}$ as the cokernel of $r_\epsilon + \ell_\epsilon$ on $[W^\bullet \otimes \mathcal{M}(\mathbb{I}_{h_i,1}) \otimes \mathcal{M}(\mathbb{I}_{h_i,1})']_{(\widetilde{C},P_{\bullet}\cup \Qpm)}$.
\endproof

\subsection{Symplectic fermions \texorpdfstring{$SF^+_d$}{SF+d}}
Defined for $d \in \ZZ_{\ge 1}$, these are  $C_2$-cofinite VOAs which are not rational, and $SF^+_1\cong \mathcal{W}(2)$. We note that symplectic fermions were first studied in the mathematical physics literature in \cite{KauschFer, GKTrip1,GKIndFusion}, and in the mathematical literature in \cite{AbeNotRational}, where it was shown that there are four inequivalent simple $SF(d)^+$-modules, denoted $SF(d)^{\pm}$, $SF(d)_\theta^{\pm}$, and two non-simple indecomposable modules $\widehat{SF(d)}^{\pm}$. By \cite{SympFermZhu}, Zhu's algebra
is $2^{2d-1}+8d^2+1$ dimensional and has bimodule decomposition 
\begin{equation}\label{eq:BDF}
A(SF(d)^+)= \bigoplus_{i=1}^4B_i,
\end{equation} with components:
\begin{enumerate}
\item[(1)] $B_1 \cong \mathbb{C}\cong (SF(d)_\theta^{+})_0 \otimes (SF(d)_\theta^{+})_0^{\vee}$.
\item[(2)] $B_2 \cong M_{2d}(\mathbb{C}) \cong SF(d)^-_0\otimes (SF(d)^-_0)^\vee$ 
\item[(3)] $B_3 \cong M_{2d}(\mathbb{C}) \cong (SF(d)^-_\theta)_0 \otimes (SF(d)^-_\theta)_0^\vee$;  and
\item[(4)]  $B_4\cong \Lambda^{\text{even}} V_{2d}$ is reducible and indecomposable, equal to the degree zero component of the reducible indecomposable $SF(d)^+$-module 
$\widehat{SF(d)}^{+}$.
 \end{enumerate}

\noindent Consider now a stable pointed curve $(C,P_{\bullet})$ as in \S\ref{sec:main}; that is the curve $C$ has one node $Q$, we denote by $\eta \colon \widetilde{C}\to C$ its normalization, and $\Qpm=\eta^{-1}(Q)$. For simple components
$\mathscr{S}=\{ SF(d)^-_0\otimes (SF(d)^-_0)^\vee, (SF(d)_\theta^{+})_0 \otimes (SF(d)_\theta^{+})_0^{\vee}, (SF(d)^-_\theta)_0 \otimes (SF(d)^-_\theta)_0^\vee \}$ in \eqref{eq:BDF} 
and $\SV:=\{S_0 \otimes S_0^\vee  \in \mathscr{S} \text{ such that } \mathcal{M}(S_0)=L(S_0)\}$,
by Corollary \ref{cor:FactorSum}
\begin{multline}\label{eq:SympRes}[W^\bullet]_{\LCP}
\cong
\bigoplus_{S_0 \otimes S_0^\vee \in \SV}[W^\bullet \otimes S  \otimes S']_{(\widetilde{C},P_{\bullet}\cup \Qpm)} \\
\oplus \bigoplus_{S_0 \otimes S_0^\vee \in \mathscr{S} \setminus \SV} [W^\bullet \otimes \mathcal{M}(S_0)\otimes \mathcal{M}(S_0^\vee)]_{(\widetilde{C},P_{\bullet}\cup \Qpm)} \, \oplus  \,
  [W^\bullet \otimes \Phi ( \Lambda^{\text{even}} V_{2d}) ]_{(\widetilde{C},P_{\bullet}\cup \Qpm)}.\end{multline}

\section{Questions}\label{sec:Questions}
To summarize: Representations of strongly rational VOAs have a number of important properties.  This class can be compared with modules over strongly finite VOAs, and in this work we begin to build infrastructure which may be useful in an investigation of whether the latter also define vector bundles on $\overline{\mathcal{M}}_{g,n}$.  In this section,  we formalize this and two more questions about these sheaves. For motivation, we briefly describe what is known about the generalized Verlinde bundles, and list a few points of comparison between the two types of VOAs (noting that the picture is much fuller and more interesting, as can be seen in \cite{CGPitch}).  

By \cite{tuy, NT, dgt2}, if $V$ is $C_2$ cofinite and rational, then $\mathbb{V}(V;W^\bullet)$ is a vector bundle. By \cite{mop,moppz, dgt3}, if $V$ is strongly rational (so $V$ is also simple and self-dual), then one may give explicit formulas for Chern classes, showing they are tautological using the methods of cohomological field theories.  In certain cases, the bundles are globally generated, and so Chern classes are base point free \cite{fakhr, DG}.

Representations of strongly rational VOAs have many other important properties, lending to a comparison with the strongly rational case.  For instance:

\begin{enumerate}
    \item[(A)]  $V$-modules are objects of a modular tensor category (MTC)   \cite{HL, HuangFusionModular}.
    \item[(B)] In the language of \cite{MooreSeiberg},  rational conformal field theories (RCFTs) are determined by a coherent sheaf of coinvariants (and dual sheaf of conformal blocks), and of \cite{FS}, by vector bundles of coinvariants together with their projectively flat connection.
    \item[(C)] Properties of the MTC from (A) correspond to those of sheaves of (B)  \cite{bk}.
\end{enumerate}
 A modular tensor category is a braided tensor category with additional structure (see \cite{T}, and in this context  \cite[\S 2.6]{CGPitch}). If $V$ is strongly rational, then every $V$-module is ordinary, and can be expressed as a finite sum of simple $V$-modules $S^i$. By \cite{ZhuGlobal}, fusion coefficients  \begin{equation}\label{FusionRing}W^i\otimes W^j=\mathcal{N}^k_{ij} \ W^k,\end{equation} are determined by the dimensions of vector spaces of conformal blocks on $\overline{\mathcal{M}}_{0,3}$ \[\mathcal{N}^k_{ij}={\rm{dim}}\left(\mathbb{V}_0(V; (W^i, W^j, (W^k)')^{\vee}\right) \in \mathbb{Z}_{\ge 0}.\]  Equation \eqref{FusionRing} gives the product structure on the fusion ring  ${\rm{Fus}}^{Simp}(V)=\text{Span}_{\mathbb{Z}}\{S^i\}$ spanned by (isomorphism classes) of simple $V$-modules (with unit element $V$).

A $C_2$-cofinite but not rational VOA has at least one indecomposable but not simple module (such a VOA or conformal field theory is called logarithmic). Let $V$ be strongly finite (so as in \ref{sec:FiniteDefs}, $V$ is $C_2$-cofinite, simple, and self-dual).
\begin{enumerate}
    \item[(A')] The braided tensor category $Mod^{g  r}(V)$  \cite{HLZ} is conjecturally log-modular \cite{CGPitch}. 
    \item[(B')] Logarithmic conformal field theories (LCFTs) are determined by finite dimensional vector spaces of coinvariants and conformal blocks defined by $V$-modules \cite{CGPitch}.
    \item[(C')] Features of (A') and (B') correspond to properties of the sheaves $\mathbb{V}(V;W^\bullet)$. 
\end{enumerate} 

Log modular categories are braided tensor categories with certain additional structure (see \cite[Definition~3.1]{CGPitch}).  
If $V$ is strongly finite, the important modules are the (finitely many) projectives $\{P_i\}_{i\in I}$, which consist of the simple modules and their projective covers. By \cite[Proposition 3.2, part (d)]{CGPitch}, if \cite[Conjecture 3.2]{CGPitch} holds, then these projective $V$-modules form an ideal in $Mod^{g \cdot r}(V)$, closed under tensor products, and taking contragredients, with fusion coefficients given in terms of dimensions of vector spaces of conformal blocks.

\smallskip

\noindent
Given these analogies, it is natural to ask the following  questions.

\begin{question}\label{Q1}Given  $W^1, \ldots, W^n$ modules over a $C_2$-cofinite VOA $V$, is $\mathbb{V}(V;W^\bullet)$ a vector bundle? If not, what additional assumptions are necessary so it is?
\end{question}

  By Corollary \ref{cor:CoherentMain}, if $V$ is $C_2$ cofinite, then $\mathbb{V}(V;W^\bullet)$ is coherent.  Question \ref{Q1} asks then
  whether $\mathbb{V}(V;W^\bullet)$ is locally free. For $V$ rational and $C_2$-cofinite, the proof that $\mathbb{V}(V,W^\bullet)$ is locally free relies on the sewing theorem \cite[Theorem 8.5.1]{dgt2} (see \cite[VB Corollary]{dgt2}).  Different notions of sewing theorems exist in related contexts. For instance, in \cite{consistency}, in studying local conformal field theories, the authors aim to identify spaces of correlation functions, which among other things, are compatible with the operation of sewing of Riemann surfaces.  Conditions necessary to recognize conformal field theories generally include some sort of functoriality with respect to the sewing of surfaces (see \cite{consistency} and references therein). An analogous sewing statement has been proved for curves of genus zero and one by Huang in \cite{HuangLog}.  As explained  in \cite{ZhuMod}, according to the physical arguments described by Moore and Seiberg in \cite{MooreSeiberg}, to construct a conformal field theory based on representations of VOAs, a key consistency condition is the modular invariance of the characters of irreducible representations of $V$. In \cite{ZhuMod}, this was shown for strongly rational VOAs, and in \cite{MiyamotoC2} an analogous statement was shown in case $V$ is strongly finite. 

\begin{question}\label{Q2} (a) What are the Chern classes of $\mathbb{V}(V;W^\bullet)$? (b) Are they tautological?
\end{question}

In case $V$ is strongly rational, then by \cite[Theorem 1]{dgt3}, the collection consisting of the Chern characters of all vector bundles of coinvariants forms a semisimple cohomological field theory, giving rise to explicit expressions for Chern classes (see \cite[Corollaries 1 and 2]{dgt3}).  This was proved following the original result for Verlinde bundles \cite{mop,moppz}.

If $V$ is strongly finite, but not rational, while one still has Chern characters (with $\mathbb{Q}$-coefficients) an analogous CohFT is not obviously available. For a semisimple CohFT, one naturally obtains the structure of a Fusion ring, which is necessarily semi-simple. In the strongly finite, non-rational case, there are three options for what could play the role of a fusion ring, including ${\rm{Fus}}^{Simp}(V)$ spanned over $\mathbb{Z}$ by the projective modules. For example the  $4p-2$ simple and indecomposable  $\mathcal{W}(p)$-modules discussed in \S \ref{ex:Triplet} are closed under tensor products and  their $\mathbb{Z}$-span forms  ${\rm{Fus}}^{Simp}(\mathcal{W}(p))$, but this ring is not semi-simple.   One may therefore need new ideas for computing Chern classes.  As factorization presentations are quotients, it seems unlikely that classes will be generally tautological.

\begin{remark} Questions have been asked about sheaves defined by rational and $C_2$-cofinite VOAs which may be relevant to the more general situation considered here.  For instance,  see \cite[\S 6]{dgt3} and \cite[\S 9]{DG}.
\end{remark}
 
 \section{Appendix: Decompositions of finite semi-simple associative algebras}\label{sec:SaltyBell}
 If there is a bimodule decomposition of $A(V)$ as a sum $X_\ell \otimes X_r$ of tensor products of simple left $A(V)$-modules $X_\ell$ and right $A(V)$-modules $X_r$, which are not necessarily dual, then one may apply Theorem \ref{thm:FactorizationRes} to obtain a factorization of nodal coinvariants. Such a factorization would be as in \cite{tuy, NT, dgt2}.  Proposition \ref{prop:SB} says that a finite such decomposition can be found if and only if $A(V)$ is both finite dimensional and semisimple, and in this case it will follow that $X_\ell$ and $X_r$ are dual. This is a generalization of the well-known statement about when rings are a direct sum of proper two-sided ideals.  We could not find it in the literature, and so we have included it here for completeness. The essential ideas of the proof were communicated to us by Jason Bell and David Saltman.

 \begin{prop}\label{prop:SB}Let $k$ be a field. An associative $k$-algebra $R$ has a direct sum bimodule decomposition $R \cong \bigoplus_i \ (X_\ell^i \otimes X_r^i)$, in terms of left and right $R$-modules $X_\ell^i$ and $X_r^i$, if and only if $R$ is finite dimensional over $k$ and semi-simple.
 \end{prop}
\proof Suppose we have such a decomposition $R = \bigoplus_i (X_\ell^i \otimes X_r^i)$ and write $I_i = X_\ell^i \otimes X_r^i$. Then $I_i$ are sub-bimodules and hence 2-sided ideals of $R$. As $I_i I_j \subset I_i \cap I_j = 0$ for $i \neq j$, it follows that the ring structure in $R$ can be computed componentwise:
\[ (\sum_i x_i)(\sum_i y_i) = \sum_i x_i y_i \text{ for $x_i, y_i \in I_i$.}\]
Consequently, it follows that if we write $1 = \sum e_i$ with $e_i \in I_i$, then the $I_i$ are rings with unit $e_i$ and that the algebra $R$ is a product of the algebras $I_i$. To check $R$ is semisimple, it  suffices to show  each of the $I_i$ are semisimple. We may therefore reduce to the case that $R = I_i$ -- that is, that $R$ can be written as $X_\ell \otimes X_r$ for a right module $X_r$ and a left module $X_\ell$.

So, assume that we have $R \cong X_\ell \otimes X_r$ as above. We will show that $R$ has no nontrivial 2-sided ideals, and is therefore simple and hence as it is also finite dimensional. If $J \triangleleft R$ is such an ideal, note that $J = RJ = JR$ and so $X_\ell \otimes X_r = JX_\ell \otimes X_r = X_\ell \otimes X_r J$. If $JX_\ell = X_\ell$ then $J X_\ell \otimes X_r = X_\ell \otimes X_r$ and so $J = JR = R$, contridicting the nontriality of $J$. So we conclude $J X_\ell \neq X_\ell$ and similarly that $X_r \neq X_r J$.
If $JX_\ell = 0$ then we may similarly conclude that $J = JR = 0$. So we find that $JX_\ell \neq 0 \neq X_r J$. Consequently, we may choose $0 \neq x \in JX_\ell$ and $y \in X_r \setminus X_r J$. But now we find $x \otimes y \in JX_\ell \otimes X_r = X_\ell \otimes X_r J$. But as $y \not\in X_r J$, this is impossible. Indeed, considering, as vector spaces, the short exact seq uence
\[\xymatrix{0 \ar[r] & X_r J \ar[r] & X_r \ar[r] & X_r / X_r J \ar[r] & 0 }\]
which remains exact upon tensoring with $X_\ell$
\[\xymatrix{0 \ar[r] & X_\ell \otimes X_rJ \ar[r] & X_\ell \otimes X_r \ar[r] & X_\ell \otimes X_r/X_rJ \ar[r] & 0}\]
and as the element $y$ has nonzero image in $X_r / X_r J$, the element $x \otimes y$ must have nonzero image in $X_\ell \otimes X_r/X_r J$, showing that it cannot be in $X_\ell \otimes X_r J$.\endproof

\subsection*{Acknowledgements}
We thank Jason Bell and David Saltman for proofs of Proposition \ref{prop:SB}. Thanks to 
Dra{\v z}en Adamovi{\'c}, David Ben-Zvi, Thomas Creutzig, Bin Gui, Andr{\'e} Henriques, Yi-Zhi Huang, Haisheng Li, and Anton Milas, for helpful comments and conversations.  Gibney was supported by NSF DMS--1902237, and Krashen was supported by NSF DMS -- 1902144.

\bibliographystyle{alphanum}
\bibliography{BiblioFinite}

\end{document}